\documentclass{ws-jktr}
\usepackage{graphicx,epsfig}
\usepackage{tikz}
\usetikzlibrary{decorations.markings,matrix}
\usetikzlibrary{knots}
\usepackage{verbatim}
\usepackage[all,import,rotate]{xy}
\xyoption{all}

\begin{document}

\markboth{Zhiqing Yang}
{Instructions for Typesetting Camera-Ready Manuscripts}

\title{Regional knot invariants}

\author{Zhiqing Yang}

\address{School of Mathematical Sciences, Dalian University of Technology, P.R.China}

\maketitle

\begin{abstract}
In this paper, a regional knot invariant is constructed. Like the Wirtinger presentation of a knot group, each planar region contributes a generator, and each crossing contributes a relation. The invariant is call a tridle of the link. As in the quandle theory, one can define Alexander quandle and get Alexander polynomial from it. For link diagram, one can also define a linear tridle and its presentation matrix. A polynomial invariant can be derive from the matrix just like the Alexander polynomial case.
\end{abstract}



\section{Introduction}	

Knots and links can be represented by planar diagrams. Knot invariants can be constructed from the planar diagrams. For example, in 1928, J.W. Alexander {\cite{A}} discovered the famous Alexander polynomial. It has many connections with other topological invariants. More than 50 years later, in 1984 Vaughan Jones {\cite{J}}  discovered the Jones polynomial. Soon, the HOMFLYPT polynomial {\cite{H}} and Kauffman polynomial were found. They all can be defined using planar diagrams of links {\cite{K}}. Other invariants like knot groups, quandles are also related to planar diagrams. In most of the constructions, one uses the combinatorial information of the overcrossing arcs, undercrossing arcs and crossings. For example, the Wirtinger presentation of knot group assign one generator for each arc, one relation for each crossing {\cite{R}}.

This paper introduces some regional invariants. A link diagram divides the plane into many regions, then for each region one can introduce a generator, and each crossing one can introduce a relator. Then under certain conditions one can get a knot invariant. In this case, a relator contains four generators, any three  generators uniquely determine the fourth generator. This is a ternary relation. We call the invariant a tridle.

Tridles are connected to many other knot invariants, for example, knot quandle. Their construction are similar. Furthermore, knot quandle has a linear version. It is called the Alexander quandle. One can also get a presentation matrix for the Alexander quandle. For different link diagrams one get different matrices. Then one can
introduce some equivalence relation. The equivalence class of such matrices is a knot invariant. One can also define elementary ideals from presentation matrices, and get the Alexander polynomial. Similar things happen for tridles. Tridle also has a linear version. We call it a linear tridle. One can then get a presentation matrix for a linear tridle. For different link diagrams one get different matrices. Then some equivalence relation are introduced. The equivalence class of such matrices is also a knot invariant. One can then define elementary ideals from presentation matrices. However, since the matrix has two variables $x,y$, the polynomial is a two variable polynomial.  Calculation show for many easy knots the polynomial is a polynomial in $xy$. We do not know whether this is always true. The topological meaning of the polynomial is not clear to the author.

\section{Regional invariant}

We propose the following construction. A planar knot diagram $D$ divides the plane into many regions. Like in the Wirtinger presentation case, we assign a generator for each region, and at each crossing, we add one relation. An easy choice is a linear equation. In Fig. \ref{fig1}, if the regions are assigned generators $a,b,c,d$, the relation at the crossing is $r_1: a+bx+cxy+dy=0$, where $x,y$ are two fixed  variables. The resulting algebraic structure is denoted as $LT(D)=\{a,b,c,d \cdots \mid r_1,r_2 \cdots  \}$. We call it a linear tridle of $D$. One can try the more general equation $ax+by+cz+dw=0$, however, calculation (Reidemeister invariance) shows that the condition $xz=yw$ must be satisfied. Then it is equivalent to our setting. Alexander originally used $ax-bx+c-d=0$ {\cite{A}}, which is a special case of our choice.

\begin{figure}[th]
\begin{tikzpicture}[scale=.75,mine/.style={font=\itshape\large}]
\draw[thick]  [->] (0,0) -- (4,0);
\draw[thick]   (2,-2) -- (2,-0.3);
\draw[thick]   [->] (2,0.3) -- (2,2);
\draw (3,1) node[mine] {$a$};
\draw (1,1) node[mine] {$b$};
\draw (1,-1) node[mine] {$c$};
\draw (3,-1) node[mine] {$d$};

\draw[thick]  [->] (5,0) -- (9,0);
\draw[thick]   [<-] (7,-2) -- (7,-0.3);
\draw[thick]   (7,0.3) -- (7,2);
\draw (8,1) node[mine] {$b$};
\draw (6,1) node[mine] {$a$};
\draw (6,-1) node[mine] {$d$};
\draw (8,-1) node[mine] {$c$};
\end{tikzpicture}
\caption{New notations in this paper.}
\label{fig1}
\end{figure}
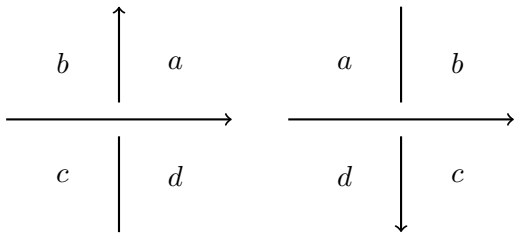

In the equation $a+bx+cxy+dy=0$, let $x,y$ be invertible. Then if one has any three generators of the set $\{a,b,c,d\}$, the fourth generator will be uniquely determined. Fixing a field $S$, $LT(D)$ is defined to be the minimal set satisfies the following properties.

(1) Each region contributes a generator $a_i$, and $LT(D)$ contains all the $a_i$.

(2) The generators satisfy the relations defined at the crossings.

(3) (Closure property) For any equation $\alpha_1+x\alpha_2 +xy\alpha_3 +y \alpha_4=0$, if  three of the $\alpha_i $ lies in $ LT(D)$, then the solution of fourth lies in p$ LT(D)$.

Such $LT(D)$ exists and can be constructed as follows. Let $S[x,x^{-1},y,y^{-1}]$ be the Laurent polynomial ring. Denote $M=S[x,x^{-1},y,y^{-1}](a_1,a_2,\cdots )$ the free module generated by $a_1,a_2,\cdots$, and $T$ be the collection of the relations. $N=M/T$ denote the quotient abelian group. Without lose of generality, the image of $a_i$ will be still called $a_i$. Inside $N$, let $R_1=\{a_1,a_2, \cdots \}$. Suppose $R_n$ is defined. For any equation $\alpha_1+x\alpha_2 +xy\alpha_3 +y \alpha_4=0$, if  three of the $\alpha_i \in \cup_{i=1}^n R_i$, then $R_{n+1}$ is the collection of the solution of the left $\alpha_j$. $R_{n+1}$ is the set of those solutions, and is a finite set. Then $LT(D)=\cup_{i=1}^{\infty} R_i$.

From the equation $\alpha_1+x\alpha_2 +xy\alpha_3 +y \alpha_4=0$, we can solve $\alpha_1=F_1(\alpha_2,\alpha_3,\alpha_4)$, $\alpha_2=F_2(\alpha_3,\alpha_4,\alpha_1)$,  $\alpha_3=F_3(\alpha_4,\alpha_1,\alpha_2)$ and $\alpha_4=F_4(\alpha_1,\alpha_2,\alpha_3)$. A map $f:LT\to LT'$ is a homomorphism if $f(F_k(a,b,c))=F_k(f(a),f(b),f(c))$ for any $a,b,c\in LT$, $k=1,2,3,4$. Isomorphism can be defined is the usual way.

Like the Tietze transformation for groups, one can define Tietze transformation for presentation of linear tridles.

\begin{definition}
Given a tridle $LT=\{X  \mid R  \}$, one can apply the following operations.

\noindent {\bf T1:} Adding a superfluous relation

$LT=\{X  \mid T  \}$ becomes  $LT'=\{X  \mid T \cup \{r \} \}$, where  $r\in N(T)$ the normal closure of the relations, i.e., $r$ is a consequence of $T$;

\noindent {\bf T2:} Removing a superfluous relation

$LT'=\{X  \mid T \cup \{r \} \}$ becomes $LT=\{X  \mid T  \}$ where $r\in N(T)$, i.e., $r$ is a consequence of $T$;

\noindent {\bf T3:} Adding a superfluous generator

$LT=\{X  \mid T \}$ becomes $LT'=\{X'  \mid T'  \}$, where $X' = X\cup \{ z\}$, $z$ is a new symbol not in $X$, and $T' = T\cup\{z=F_k(\alpha_1,\alpha_2,\alpha_3) \}$, where $k\in \{1,2,3,4\}$ and $\alpha_i \in R$;

\noindent {\bf T4:} Removing a superfluous generator

$LT'=\{X'  \mid T'  \}$ becomes $LT=\{X  \mid T \}$, where $X' = X - \{ z\}$, and $T' = T\cup\{z=F_k(\alpha_1,\alpha_2,\alpha_3) \}$, where $k\in \{1,2,3,4\}$ and $\alpha_i \in R$.

\end{definition}

\begin{lemma}
Tietze transformations do not change a tridle.
\end{lemma}

\begin{proof}
In the construction of a linear tridle, the quotient groups $N=M/T$ and $N'=M'/T'$ are isomorphic.

For $T_1$ or $T_2$ invariance, in each step, $R_i\cong R_i'$, hence $LT\cong LT'$.

For $T_3$ invariance, in the first step, $R_1\subset R_1'$, hence $R_i\subset R_i'$ and $LT\subset LT'$. However, $z=F_k(\alpha_1,\alpha_2,\alpha_3)$ and $\alpha_i \in R$. Suppose $\alpha_i \in R_{k_i}$, and $n>k_i$. Then  $\alpha_i \in R_n$, hence $z\in R_{n+1}$. So we have $R_1'\subset R_{n+1}$, hence $R_i'\subset R_{n+i}$ and $LT'\subset LT$. Therefor $LT=LT'$.

$T_4$ is the inverse of $T_3$, the proof of invariance is the same.
\end{proof}

\begin{theorem}
$LT(D)$ is a knot invariant.
\end{theorem}

\begin{proof}
We check the Reidemeister moves one by one. According to M. Polyak {\cite{P}}, for oriented link diagrams, we need to check Reidemeister move one, and Reidemeister move two and three as follows.

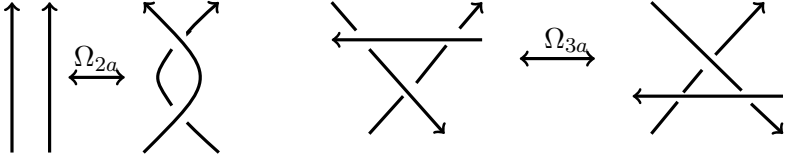
\begin{figure}[h]
\begin{tikzpicture}[scale=.25,mine/.style={font=\itshape\large}]
\draw[very thick]  [->] (-7,0) -- (-7,8);
\draw[very thick]  [->] (-5,0) -- (-5,8);

\draw[very thick]  [<->] (-4,4) -- (-1,4);
\draw (-2.5,5) node[mine] {$\Omega_{2a}$};

\draw[very thick] [<-] (0,8) .. controls (4,4) and (4,4) .. (0,0);
\draw[very thick] [<-] (4,8) .. controls (2.5,6.7) and (2,6)  .. (2.5,6.5);
\draw[very thick]  (1.5,5.5) .. controls (0.5,4) and (0.5,4)  .. (1.5,2.5);
\draw[very thick]  (4,0) .. controls (3,0.9)   .. (2.3,1.6);

\draw[very thick]  [<-] (10,6) -- (18,6);
\draw[very thick]   (10,8) -- (11.3,6.5);
\draw[very thick]  [->] (12,5.5) -- (16,1);

\draw[very thick][<-] (18,8) -- (16.8,6.6);
\draw[very thick] (16.1,5.7) -- (14.5,3.7);
\draw[very thick] (13.8,3) -- (12,1);

\draw[very thick]  [<->] (20,5) -- (24,5);
\draw[very thick] (22.5,6) node[mine] {$\Omega_{3a}$};

\draw[very thick][<-] (26,3) -- (34,3);
\draw[very thick] (27,8) -- (31.8,3.3);
\draw[very thick] [->]  (32.3,2.6) -- (34,1);
\draw[very thick][<-] (33,8) -- (30.4,5.2);
\draw[very thick] (29.7,4.4) -- (28.7,3.2);
\draw[very thick] (28.4,2.7) -- (27,1);
\end{tikzpicture}
\caption{Reidemeister move two and three.}
\label{fig2}
\end{figure}

First, for Reidemeister move one invariance, suppose we have knot diagrams $D,D'$ as in Fig. \ref{fig3}. The regions are assigned with generators $a,b,c,\cdots $. $LT(D)=\{a,b, \cdots \mid r_1,r_2 \cdots  \}$. Then in $D'$, the generators satisfy $b+ax+bxy+cy=0$, hence $LT(D')=\{c,a,b, \cdots \mid b+ax+by+cxy=0, r_1,r_2 \cdots  \}$. By $T_3$ transformation invariance, $LT(D') \cong LT(D)$. For other cases of Reidemeister move one, the proofs are the same.

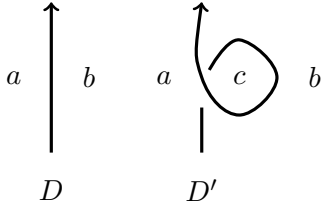
\begin{figure}[h]
\begin{tikzpicture} [scale=.5,mine/.style={font=\itshape\large}]

\draw[very thick] [<-] (-1,2) -- (-1,-2);
\draw[very thick] (-2,0) node[mine] {$a$};
\draw[very thick] (0,0) node[mine] {$b$};
\draw[very thick] (-1,-3) node[mine] {$D$};
\draw[very thick] [<-] plot[smooth, tension=.7] coordinates {(3,2) (3,0) (4,-1) (5,0) (4,1) (3.2,0.2)};
\draw[very thick]  plot[smooth, tension=.7] coordinates {(3,-0.8) (3,-2)};
\draw[very thick] (2,0) node[mine] {$a$};
\draw[very thick] (6,0) node[mine] {$b$};
\draw[very thick] (4,0) node[mine] {$c$};
\draw[very thick] (3,-3) node[mine] {$D'$};
\end{tikzpicture}
\caption{Reidemeister move one invariance.}
\label{fig3}
\end{figure}

For Reidemeister move two invariance, suppose we have knot diagrams $D,D'$ as in Fig. \ref{fig4}. The regions are assigned with generators $a,b,c,\cdots $. $LT(D)=\{a,b, c, \cdots \mid r_1,r_2 \cdots  \}$. Then in $D'$, at the lower crossing, we have relation $w+ax+b'xy+cy=0$, and at the upper crossing we have $w+ax+bxy+cy=0$. This implies that $b=b'$. Since $w=-ax-bxy-cy$ and \begin{gather*} LT(D')=\{w,b',a,b,c \cdots \mid w+ax+bxy+cy=0,w+ax+b'xy+cy=0, r_1,r_2 \cdots  \}, \end{gather*} we can use $T_3$ transformation to delete $w,b'$ and those two relations. $LT(D') \cong LT(D)$.

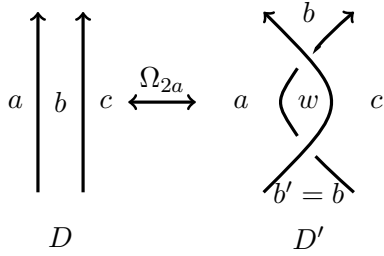
\begin{figure}[h]
\begin{tikzpicture}[scale=0.3,mine/.style={font=\itshape\large}]
\draw[very thick]  [->] (-10,0) -- (-10,8);
\draw[very thick]  [->] (-8,0) -- (-8,8);

\draw[very thick]  [<->] (-6,4) -- (-3,4);
\draw (-4.5,5) node[mine] {$\Omega_{2a}$};
\draw (-9,-2) node[mine] {$D$};
\draw (-11,4) node[mine] {$a$};
\draw (-9,4) node[mine] {$b$};
\draw (-7,4) node[mine] {$c$};

\draw[very thick] [<-] (0,8) .. controls (4,4) and (4,4) .. (0,0);
\draw[very thick] [<-] (4,8) .. controls (2.5,6.7) and (2,6)  .. (2.5,6.5);
\draw[very thick]  (1.5,5.5) .. controls (0.5,4) and (0.5,4)  .. (1.5,2.5);
\draw[very thick]  (4,0) .. controls (3,0.9)   .. (2.3,1.6);

\draw (2,-2) node[mine] {$D'$};
\draw (-1,4) node[mine] {$a$};
\draw (2,4) node[mine] {$w$};
\draw (5,4) node[mine] {$c$};
\draw (2,0) node[mine] {$b'=b$};
\draw (2,8) node[mine] {$b$};

\end{tikzpicture}
\caption{Reidemeister move two invariance.}
\label{fig4}
\end{figure}

For Reidemeister move three invariance, suppose we knot diagrams $D,D'$ as in Fig. {\ref{fig5}}. The regions are assigned with generators $X,Y,Z,\cdots $. $LT(D)=\{X,Y,Z,W,U,V \cdots \mid r_1,r_2 \cdots  \}$.

For $D$, there are three crossings hence we have three relations \begin{gather*} r_1:xyX+yS+xY+Z=0, \\ r_2: xyX+yS+xW+U=0,\\ r_3:xyZ+yS+xV+U=0.\end{gather*} Regard $S,U,V$ as unknowns, we get \begin{gather*}S = -(Xxy+Yx+Z)/y,\\ U = -Wx+Yx+Z,\\ V = Xy-Zy+W. \end{gather*}

Since $U,V,S$ are determined by $X,Y,Z,W$, by Tietze transformation, we can delete $S, r_1$, in the other relations replace $S$ by $-(Xxy+Yx+Z)/y$, then $r_2,r_3$ become \begin{gather*} \overline{r}_1: U = -Wx+Yx+Z,\\  \overline{r}_2: V = Xy-Zy+W. \end{gather*} Then we get a new presentation from $LT(D)=\{X,Y,Z,W,U,V \cdots \mid r_1,r_2 \cdots  \}$ by deleting $S, r_1, r_2,r_3$ and adding $\overline{r}_1, \overline{r}_2$.

Similarly, for $D'$, there are three crossings hence we have three relations \begin{gather*} r_1': Wxy+Tx+U'y+V'=0,\\ r_2': Yxy+Tx+Zy+V'=0,\\ r_3': Yxy+Tx+Xy+W=0. \end{gather*} Regarding $S,U',V'$ as unknowns, we get \begin{gather*} S = -(Yxy+Xy+W)/x,\\ U' = -Wx+Yx+Z,\\ V' = Xy-Zy+W. \end{gather*} Notice that the formulas for $U', V'$ and $U,V$ are the same. Hence we can also get a new presentation from $LT(D')=\{X,Y,Z,W,U',V' \cdots \mid r_1',r_2' \cdots  \}$ by deleting $S, r_1', r_2',r_3'$ and adding $\overline{r}_1', \overline{r}_2'$, where \begin{gather*} \overline{r}_1': U' = -Wx+Yx+Z, \\ \overline{r}_2': V' = Xy-Zy+W. \end{gather*}

Since after Tietze transformations $LT(D)$ and $LT(D')$ have the same presentation, they are isomorphic.

\begin{figure}[h]
\begin{tikzpicture}[scale=0.3,mine/.style={font=\itshape\large}]
\draw[very thick]  [<-] (10,6) -- (18,6);
\draw[very thick]   (10,8) -- (11.3,6.5);
\draw[very thick]  [->] (12,5.5) -- (16,1);

\draw[very thick][<-] (18,8) -- (16.8,6.6);
\draw[very thick] (16.1,5.7) -- (14.5,3.7);
\draw[very thick] (13.8,3) -- (12,1);

\draw[very thick] (14,-1) node[mine] {$D$};
\draw[very thick] (9,7) node[mine] {$X$};
\draw[very thick] (14,7) node[mine] {$Y$};
\draw[very thick] (19,7) node[mine] {$Z$};
\draw[very thick] (11,4) node[mine] {$W$};
\draw[very thick] (14,5) node[mine] {$T$};
\draw[very thick] (17,4) node[mine] {$V$};
\draw[very thick] (14,1) node[mine] {$U$};

\draw[very thick]  [<->] (20,5) -- (24,5);
\draw[very thick] (22.5,6) node[mine] {$\Omega_{3a}$};

\draw[very thick][<-] (26,3) -- (34,3);
\draw[very thick] (27,8) -- (31.8,3.3);
\draw[very thick] [->]  (32.3,2.6) -- (34,1);
\draw[very thick][<-] (33,8) -- (30.4,5.2);
\draw[very thick] (29.7,4.4) -- (28.7,3.2);
\draw[very thick] (28.4,2.7) -- (27,1);
\draw[very thick] (30,-1) node[mine] {$D$};
\draw[very thick] (27,6) node[mine] {$X$};
\draw[very thick] (30,7) node[mine] {$Y$};
\draw[very thick] (33,6) node[mine] {$Z$};
\draw[very thick] (30.2,3.8) node[mine] {$S$};
\draw[very thick] (26,2) node[mine] {$W$};
\draw[very thick] (30,1) node[mine] {$U'$};
\draw[very thick] (34,2) node[mine] {$V'$};
\end{tikzpicture}

\caption{Reidemeister move three invariance.}
\label{fig5}
\end{figure}
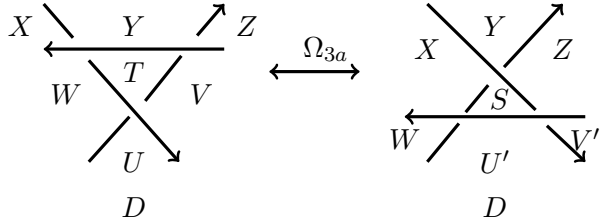

\end{proof}

The above construction is similar to the Alexander quandle. Likewise, we can also get a matrix invariant. From a presentation of $LT(D)=\{x_1,x_2, \cdots \mid r_1,r_2 \cdots  \}$ derived from a link diagram $D$, we can get a representation matrix $M(D)$ as follows. Each column $i$ corresponds to a generator $x_i$, and each row corresponds to a relation $r_i$.

\begin{example} In Fig. {\ref{fig6}}, we get the equations \begin{gather*} b+xe+xyc+ya=0,\\ c+xe+xyd+ya=0,\\ d+xe+xyb+ya=0. \end{gather*} The representation matrix  is as follows.

\begin{figure}[h]
\begin{tikzpicture}
\foreach \brk in {0,1,2} {
\begin{scope}[rotate=\brk * 120]
\node[knot crossing, transform shape, inner sep=1.5pt] (k\brk) at (0,-1) {};
\end{scope}
}
\foreach \brk in {0,1,2} {
\pgfmathparse{int(Mod(\brk - 1,3))}
\edef\brl{\pgfmathresult}
\draw[thick,red](k\brk) .. controls (k\brk.4 north west) and (k\brl.4 north east) .. (k\brl.center);
\draw[thick,red][->]  (k\brk.center) .. controls (k\brk.16 south west) and (k\brl.16 south east) .. (k\brl);
}

\draw[thick] (0,0) node  {$a$};
\draw[thick] (0,1) node  {$b$};
\draw[thick] (0.8,-0.7) node  {$c$};
\draw[thick] (0,1) node  {$b$};
\draw[thick] (-0.8,-0.7) node  {$d$};
\draw[thick] (1.2,1.2) node  {$e$};

\end{tikzpicture}
\caption{Left hand trefoil knot.}
\label{fig6}
\end{figure}
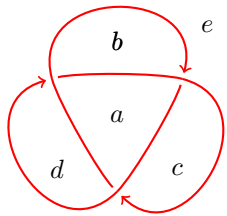

\begin{tikzpicture}
  \matrix (m) [matrix of math nodes,  nodes in empty cells]
  {
    \quad & a  & b & c & d & e  \\
    r_1   &  y &xy &   & 1 & x \\
    r_2   &  y & 1 &xy &   & x \\
    r_3   &  y &   & 1 &xy & x \\
   };
\end{tikzpicture}

\end{example}

Different link diagrams will provide different matrices. They are equivalent under the following operations.

(M1) Add a multiple of one row to another row and vice versa.

(M2) If one column contains only one nonzero invertible entry $a_{i,j}$ of the form $\pm x^ly^m$, one can simultaneously delete the ith row and jth column and vice versa.

(M3) If the ith row contains only two nonzero entries $a_{i,j}=f$ and $a_{i,k}=-f$, one can add the jth column to kth column, and vice versa.

(M4) multiply one row by $\pm x^ly^m$.

(M5) Exchange two rows (or two columns).

\begin{theorem}
The equivalence class of a representation matrix is a knot invariant.
\end{theorem}

\begin{proof}
For Reidemeister move one, suppose we have diagrams $D,D'$ as in Fig. \ref{fig3}. $D'$ has one more crossing and one more region, hence $LT(D')$ has one more generator and one more relator. So $M(D')$ has one more row. $M(D')$ can be obtained from $M(D)$ by operation M2.

For Reidemeister move two, suppose we have diagrams $D,D'$ as in figure \ref{fig4}. $D'$ has two more crossings and two more regions, hence $LT(D')$ has two more generators and two more relators.  $M(D')$ is as follows.

\begin{tikzpicture}
  \matrix (m) [matrix of math nodes,  nodes in empty cells]
  {
    \quad & a & b & b' & c & w & \cdots \\
    r_1 & x & xy &   & y  & 1 \\
    r_2 & x &  & xy  & y  & 1 \\
    \cdots   \\
   };
\end{tikzpicture}

Change $r_2$ to $r_2-r_1$, we have the following matrix.

\begin{tikzpicture}
  \matrix (m) [matrix of math nodes,  nodes in empty cells]
  {
    \quad & a & b & b' & c & w & \cdots \\
    r_1 & x & xy &   & y  & 1 \\
    r_2 &  & -xy & xy  &   &  \\
    \cdots   \\
   };
\end{tikzpicture}

Then we apply operation (M3) and get the following matrix.

\begin{tikzpicture}
  \matrix (m) [matrix of math nodes,  nodes in empty cells]
  {
    \quad & a & b  & c & w & \cdots \\
    r_1 & x & xy &  y  & 1 \\
    \cdots   \\
   };
\end{tikzpicture}

Then only one column contains $w$, we can apply operation (M2) and  obtain  $M(D)$.

For Reidemeister move two, suppose we have diagrams $D,D'$ as in Fig. \ref{fig5}. The matrices $M(D),M(D')$ are as follows.

\begin{tikzpicture}
\matrix (m) [matrix of math nodes,  nodes in empty cells]
  {
    \quad & X  & Y & Z  & W  & S & U & V & \cdots \\
    r_1   & xy & x & 1  &    & y &   & \\
    r_2   & xy &   &    & x  & y & 1 &  \\
    r_3   &    &   & xy &    & y & 1 & x \\
    \cdots   \\
   };
\end{tikzpicture}
\begin{tikzpicture}
  \matrix (m) [matrix of math nodes,  nodes in empty cells]
  {
    \quad & X  & Y & Z  & W  & T & U & V & \cdots \\
    r_1   &    &   &    & xy & x & y & 1 \\
    r_2   &    & xy& y  &    & x &   & 1 \\
    r_3   & y  & xy&    & 1  & x &   &   \\
    \cdots   \\
   };
\end{tikzpicture}

In $M(D)$, $r_3\mapsto r_3-r_2$, $r_2\mapsto r_2-r_1$. In $M(D')$, $r_1\mapsto r_1-r_2$, $r_2\mapsto r_2-r_3$

\begin{tikzpicture}
\matrix (m) [matrix of math nodes,  nodes in empty cells]
  {
    \quad & X  & Y & Z  & W  & S & U & V & \cdots \\
    r_1   & xy & x & 1  &    & y &   & \\
    r_2   &    & -x &-1  & x &   & 1 &  \\
    r_3   &-xy &   & xy & -x &   &   & x \\
    \cdots   \\
   };
\end{tikzpicture}
\begin{tikzpicture}
  \matrix (m) [matrix of math nodes,  nodes in empty cells]
  {
    \quad & X  & Y & Z  & W  & T & U & V & \cdots \\
    r_1   &    &-xy&-y  & xy &   & y &   \\
    r_2   &-y  &   & y  &-1  &   &   & 1 \\
    r_3   & y  & xy&    & 1  & x &   &   \\
    \cdots   \\
   };
\end{tikzpicture}

In $M(D)$, delete $r_1$ and column (S), $r_3\mapsto r_3/x$. In $M(D')$, delete $r_3$ and column (T), $r_1\mapsto r_1/y$.

\begin{tikzpicture}
\matrix (m) [matrix of math nodes,  nodes in empty cells]
  {
    \quad & X  & Y & Z  & W   & U & V & \cdots \\
    r_2   &    & -x &-1  & x  & 1 &  \\
    r_3   &-y &   & y & -1  &   & 1 \\
    \cdots   \\
   };
\end{tikzpicture}
\begin{tikzpicture}
  \matrix (m) [matrix of math nodes,  nodes in empty cells]
  {
    \quad & X  & Y & Z  & W   & U & V & \cdots \\
    r_1   &    &-x &-1  & x   & 1 &   \\
    r_2   &-y  &   & y  &-1   &   & 1 \\
    \cdots   \\
   };
\end{tikzpicture}

Hence they are isomorphic.

\end{proof}

In quandle theory, from a link diagram one can construct its Alexander quandle. Furthermore, one can construct presentation matrix and Alexander invariant {\cite{R}}. In the linear tridle case, we can also define some ideal invariant in a similar way. The presentation matrix is a $n\times (n+2)$ matrix. It has $(n+1)(n+2)/2$ submatrices. The greatest common divisor of determinants of those matrices is denoted by $\Delta (D)$ .

\begin{theorem}

Up to a multiple of $\pm x^ly^m$, $\Delta (D)$ is a knot invariant.
\end{theorem}

\begin{proof}
This is proved by check the moves one by one.
\end{proof}

\begin{example}  The right hand trefoil knot has presentation matrix as follows.

\begin{tikzpicture}
  \matrix (m) [matrix of math nodes,  nodes in empty cells]
  {
    \quad & a  & b & c & d & e  \\
    r_1   &  y &xy &   & 1 & x \\
    r_2   &  y & 1 &xy &   & x \\
    r_3   &  y &   & 1 &xy & x \\
   };
\end{tikzpicture}

Then $\Delta (D)$ is $1+(xy)^3$.

\end{example}

\begin{remark}
If we use a finite quotient of $S[x,x^{-1},y,y^{-1}]$,  we can also define finite linear quandles.
\end{remark}

\section{General tridle}

The linear tridle can be generalized to tridle. It can be described as follows. Given a link diagram $D$, we assign one generator for each region and one relator to each crossing. As in Fig. \ref{fig1}, the relator is an equation $F(a,b,c,d)=0$. It uniquely determines $a=F_1(b,c,d),b=F_2(c,d,a),c=F_3(d,a,b),d=F_4(a,b,c)$.

In Fig. \ref{fig5}, regard $S,U,V$ and $T,U',V'$ as unknowns, one solve the following equations.

\noindent \begin{gather*} F(Z,Y,X,S)=0,\\ F(U,W,X,S)=0,\\ F(U,V,Z,S)=0, \end{gather*} and

\noindent \begin{gather*}F(V',T,W,U')=0,\\F(V',T,Y,Z)=0,\\F(W,T,Y,X)=0. \end{gather*}

Then \begin{gather*} S=F_4(Z,Y,X),U=F_1(W,X,S)=F_1(W,X,F_4(Z,Y,X)), \\ V=F_2(Z,S,U)=F_2(Z,F_4(Z,Y,X),F_1(W,X,F_4(Z,Y,X))), \end{gather*} and

\noindent  \begin{gather*} T=F_2(Y,X,W),V'=F_1(T,Y,Z)=F_1(F_2(Y,X,W),Y,Z), \\ U'=F_4(V',T,W)=F_4(F_1(F_2(Y,X,W),Y,Z),F_2(Y,X,W),W). \end{gather*}

To get a knot invariant, we need $U=U', V=V'$. So the following conditions are satisfied.

\begin{gather}F_2(Z,F_4(Z,Y,X),F_1(W,X,F_4(Z,Y,X)))=F_1(F_2(Y,X,W),Y,Z) \\
F_1(W,X,F_4(Z,Y,X))=F_4(F_1(F_2(Y,X,W),Y,Z),F_2(Y,X,W),W).   \end{gather}

Given a link diagram, one can get a tridle $T(D)=\{a,b,c,d,\cdots \mid r_1,r_2,\cdots \}$. In general, given such a finite presentation, the tridle can be constructed as follows. Let $T_1$ be the set of generators. Suppose $T_n$ is constructed, let \begin{gather*} T_{n+1}=\{F_i(x,y,z)\mid i=1,2,3,4; x,y,z\in \cup_{i\leq n} T_i \}. \end{gather*} Then let \begin{gather*} F_2(Z,F_4(Z,Y,X),F_1(W,X,F_4(Z,Y,X)))\sim F_1(F_2(Y,X,W),Y,Z)\\ F_1(W,X,F_4(Z,Y,X))\sim F_4(F_1(F_2(Y,X,W),Y,Z),F_2(Y,X,W),W)\end{gather*}  for any $X,Y,Z,W\in \cup_{i=1}^{\infty } T_i$. Let $T(D)=\cup_{i=1}^{\infty } T_i /\sim$.

Here are some examples. The easiest solutions for $F$ are linear functions. Actually, $F(a,b,c,d)=a+bx+xyc+yd$ is the most general homogenous linear solution. One can also find non-homogenous linear solutions. For example, if $F(a,b,c,d)=a+b+c+d-k$ then we get a trivial knot invariant.

There are also nonlinear solutions, for example, $F(a,b,c,d)=ad^{-1}cb^{-1}-1$ or $F(a,b,c,d)=ad^{-1}-cb^{-1}$, we get the usual fundamental group with Dehn presentation. If $F(a,b,c,d)=ad^{-1}cb^{-1}-k$ or $F(a,b,c,d)=ad^{-1}-kbc^{-1}$, where $k$ commutes with all generators, we get a knot invariant similar to the fundamental group. However, this invariant has one extra parameter $k$.

One can also get degree three examples: $F(a,b,c,d)=ad^{-1}c-b$. But it is just the knot group. As mentioned before, there are also finite tridles. For example, a checkerboard coloring defines a tridle has two elements.

\section*{Acknowledgments}
The author is supported by the grant NFSC  No. 11271058.

\end{document}